\documentclass[11pt]{amsart}
\usepackage{}
\usepackage{amsfonts}
\usepackage{amsfonts,latexsym,rawfonts,amsmath,amssymb,amsthm}
\usepackage[plainpages=false]{hyperref}
\usepackage{mathrsfs}
\RequirePackage{amsmath} \RequirePackage{amssymb}
\usepackage{fancyhdr}
\usepackage{graphicx}
\usepackage{amscd,latexsym,amsthm,amsfonts,amssymb,amsmath,amsxtra}
\usepackage{amscd,mathrsfs,latexsym,amsmath,amsthm,amssymb,amsxtra,}
\usepackage{latexsym,amsmath,amsthm,amssymb,amsxtra,mathrsfs}
\usepackage[all]{xy}



\numberwithin{equation}{section}
\newcommand{\beq}{\begin{equation}}
\newcommand{\eeq}{\end{equation}}
\newcommand{\beqs}{\begin{eqnarray*}}
\newcommand{\eeqs}{\end{eqnarray*}}
\newcommand{\beqn}{\begin{eqnarray}}
\newcommand{\eeqn}{\end{eqnarray}}
\newcommand{\beqa}{\begin{array}}
\newcommand{\eeqa}{\end{array}}

\def\lra{\longrightarrow}

\def\bc{\begin{center}}
\def\ec{\end{center}}

\def\begeq{\begin{equation}}
\def\endeq{\end{equation}}
\def\and{\quad{\rm and}\quad}

\let\lra=\longrightarrow

\def\mapright\#1{\,\smash{\mathop{\lra}\limits^{\#1}}\,}

\newtheorem{prop}{Proposition}[section]
\newtheorem{theo}[prop]{Theorem}
\newtheorem{lem}[prop]{Lemma}

\newtheorem{cor}[prop]{Corollary}
\newtheorem{rem}[prop]{Remark}

\newtheorem{defi}[prop]{Definition}
\newtheorem{conj}[prop]{Conjecture}

\begin{document}
\bibliographystyle{plain}

\title  {Higher dimensional steady Ricci  solitons with linear curvature decay}

\date{}
\author {Yuxing $\text{Deng}^*$ }
\author { Xiaohua $\text{Zhu}^{**}$}

\thanks {*Partially supported by the NSFC Grants 11701030,
** by   the NSFC Grants 11331001 and 11771019}
\subjclass[2000]{Primary: 53C25; Secondary: 53C55,
58J05}
\keywords { Ricci flow, Ricci soliton, $\kappa$-solution, Perelman's conjecture}

\address{ Yuxing Deng\\School of Mathematical Sciences, Beijing Normal University,
Beijing, 100875, China\\
dengyuxing@mail.bnu.edu.cn}

\address{ Xiaohua Zhu\\School of Mathematical Sciences and BICMR, Peking University,
Beijing, 100871, China\\
xhzhu@math.pku.edu.cn}

\begin{abstract} We prove that  any   noncompact $\kappa$-noncollapsed steady (gradient) Ricci soliton with nonnegative curvature operator  must be rotationally symmetric if it  has a  linear curvature decay.
\end{abstract}

\maketitle

%\tableofcontents

\section{Introduction}
As one of singular model solutions of  Ricci flow,  it is important to   classify   steady (gradient) Ricci solitons under a suitable curvature condition   \cite{H1, Pe, MT}, etc..   In his celebrated paper \cite{ Pe},   Perelman conjectured that \textit{any 3-dimensional $\kappa$-noncollapsed steady (gradient) Ricci soliton  must be rotationally symmetric}.
The conjecture has been  solved by Brendle in 2012 \cite{Br1}. It is known by a result of Chen that    that any  3-dimensional ancient solution has nonnegative sectional curvature \cite{Ch}.  As a natural generalization of Perelman's conjecture  in higher dimensions,  we have

  \begin{conj}\label{higher-dim-conj}   Any $n$-dimensional  ($n\ge 4$) $\kappa$-noncollapsed steady (gradient)  Ricci soliton with positive curvature operator must be rotationally symmetric.
  \end{conj}

   An essential progress to  Conjecture \ref{higher-dim-conj} has been made  by  Brendle  in \cite{Br2}. In  fact, he proved   that   any  steady   (gradient)   Ricci soliton  with positive sectional curvature must be rotationally symmetric if  it  is  asymptotically   cylindrical.   For  $\kappa$-noncollapsed steady  K\"{a}hler-Ricci  solitons with nonnegative bisectional curvature,  the authors  recently proved  that they must be flat \cite{DZ2}, \cite{DZ4}.

\begin{defi}\label{brendle} An $n$-dimensional steady  gradient  Ricci soliton $(M,g,f)$  is called  asymptotically   cylindrical   if   the following
holds:

(i) Scalar curvature $R(x)$  of $g$ satisfies
$$\frac{C_1}{\rho(x)} \le  R(x) \le  \frac{C_2}{\rho(x)},~\forall~\rho(x)\ge r_0, $$
where $C_1, C_2$  are two positive constants and $\rho(x)$ denotes the distance from a fixed  point $x_0$.

(ii) Let $p_m$ be an arbitrary sequence of marked points going to infinity.
Consider  rescaled metrics
$ g_m(t) = r_m^{-1} \phi^*_{r_m t} g,$ where
$r_m R(p_m) = \frac{n-1}{2} + o(1)$ and $ \phi_{ t}$ is a one-parameter subgroup generated by $X=-\nabla f$.   As  $m \to\infty,$
 flows $(M,  g_m(t), p_m)$
converge in the Cheeger-Gromov sense to a family of shrinking cylinders
$(   \mathbb R \times \mathbb S^{n-1}(1), \widetilde g(t)), t \in  (0, 1).$  The metric $\widetilde g(t)$  is given by
\begin{align} \widetilde g(t) =   dr^2+ (n - 2)(2 -2t) g_{\mathbb S^{n-1}(1)},\notag
\end{align}
 where $\mathbb S^{n-1}(1)$ is the unit sphere in Euclidean space.
\end{defi}

The property (ii) above means that  for any $p_{i}\rightarrow+\infty$,
 rescaled flows $(M,$ $R(p_i)g(R^{-1}(p_i)t),p_{i})$ converge subsequently to
$(\mathbb{R}\times \mathbb{S}^{n-1},ds^2+g_{\mathbb{S}^{n-1}}(t))$ ( $t\in (-\infty,0]$),  where $g_{\mathbb{S}^{n-1}}(t))$ is a family of shrinking round shperes.     In this paper,   we  verify  the properties (i) and (ii) in Definition \ref{brendle} to show that Conjecture \ref{higher-dim-conj} is true in addition that  the scalar curvature of steady Ricci soliton has a linear curvature decay. More precisely, we have

\begin{theo}\label{main theorem}
Let $(M,g)$ be a noncompact $\kappa$-noncollapsed steady (gradient) Ricci soliton with nonnegative curvature operator. Then, it is rotationally symmetric if its scalar curvature $R(x)$ satisfies
\begin{align}\label{upper-linear-bound}
R(x)\le \frac{C}{\rho(x)}.
\end{align}

\end{theo}

In Theorem \ref{main theorem},  we do not assume that  $(M,g)$ has positive Ricci curvature.  In fact,  we can prove that
 $(M,g)$ is an Euclidean  space  if   the Ricci curvature is not strictly  positive (cf. Section 5).   Since the Euclidean  space is  rotationally symmetric,  we may assume that $(M,g)$  is not a flat space.   Then we show that  the condition  (\ref{upper-linear-bound}) in Theorem \ref{main theorem} also implies
 \begin{align}\label{lower-bound-r}R(x)\ge \frac{c_0}{\rho(x)},~\forall ~\rho (x)\ge r_0>0,
\end{align}
where $c_0>0$ is a constant (cf. Corollary \ref{cor-lower decay nonnegative case}).
  (\ref{lower-bound-r})  has been proved  for the  steady Ricci soliton  with nonnegative curvature operator and  positive Ricci curvature \cite{DZ2} (also see Theorem \ref{lower-bound-scalar-curvature}).

    Theorem \ref{main theorem} is reduced to prove

\begin{theo}\label{main theorem-2}
Let $(M,g)$ be a  $\kappa$-noncollapsed steady  (gradient)   Ricci soliton with nonnegative sectional curvature.
Suppose that  $(M,g)$
has an exactly linear curvature decay, i.e.
\begin{align}\label{linear-decay-curvature}
\frac{C_0^{-1}}{\rho(x)}\le R(x)\le \frac{C_0}{\rho(x)}
\end{align}
 for some constnat $C_0>0$. Let $g(t)=\phi_t^{\ast}g$\footnote{$\phi_t$ is defined  as in (ii) of Definition \ref{brendle}.}.   Then for any $p_{i}\rightarrow+\infty$,
 rescaled flows $(M,$ $R(p_i)g(R^{-1}(p_i)t),p_{i})$ converge subsequently to
$(\mathbb{R}\times \mathbb{S}^{n-1},ds^2+g_{\mathbb{S}^{n-1}}(t))$ ( $t\in (-\infty,0]$) in the Cheeger-Gromov topology, where
 $(\mathbb{S}^{n-1},$ $g_{\mathbb{S}^{n-1}}(t))$ is a $\kappa$-noncollapsed ancient Ricci flow with nonnegative sectional curvature. Moreover,  scalar curvature $R_{\mathbb{S}^{n-1}}(x,t)$ of $(\mathbb{S}^{n-1},g_{\mathbb{S}^{n-1}}(t))$ satisfies
 \begin{align}\label{curvature-decay-sn}
 R_{\mathbb{S}^{n-1}}(x,t)\le\frac{C}{|t|}, ~\forall~x\in \mathbb{S}^{n-1},
 \end{align}
 where $C$ is a uniform constant.
\end{theo}

 Theorem \ref{main theorem-2} is  about   steady Ricci solitons  with  nonnegative sectional curvature, which is a weaker condition  than nonnegative  curvature operator. In fact,  applying  Theorem \ref{main theorem-2} to  4-dimensional   steady Ricci solitons,  we further prove

\begin{theo}\label{cor-rotational symmetry of 4d}
 Any 4-dimensional noncompact $\kappa$-noncollapsed steady  (gradient)  Ricci soliton with nonnegative sectional curvature  must be rotationally symmetric if  it has a linear curvature decay (\ref{linear-decay-curvature}).
\end{theo}

We conjecture that  Theorem \ref{cor-rotational symmetry of 4d} holds for all dimensions. It is closely related to the classification of shrinking solitons on $\mathbb{S}^{n-1}$  with positive sectional curvature (see Section 6.1 for details). This will improve Brendle's result without the condition (ii) in Definition  \ref{brendle} \cite{Br2}.

The main step in the proof of Theorem \ref{main theorem-2} is to estimate the  diameter  of level sets of steady Ricci soliton (cf. Section 3). We study a metric flow of level sets.  This flow is very similar to the  Ricci flow. Then we can use Perelman's argument to estimate the distance functions in level sets \cite{Pe}.   In Section 4,  we  prove Theorem \ref{main theorem-2}  for steady Ricci solitons with positive Ricci curvature  by constructing a parallel vector field in  a  limit space  as in \cite{DZ2, DZ5, DZ3}.
The  proof of Theorem \ref{main theorem-2} will be completed in Section 5.
The proofs of  Theorem \ref{main theorem} and \ref{cor-rotational symmetry of 4d} are given in Section 6.

\section{Previous  results }

In this section, we recall some results for  steady Ricci solitons  in \cite{DZ2,DZ5}.  $(M,g,f)$ is called a  steady  gradient  Ricci soliton if  Ricci curvature  ${\rm Ric}(g)$  of $g$ on $M$  satisfies
\begin{align}
{\rm Ric}(g)={\rm Hess} f,
\end{align}
where $f$ is a smooth function on $M$.
We  always  assume that ${\rm Ric}(g)>0$ and there is  an equilibrium point $o$ in $M$ from this section to section 4.   The latter  means that  $\nabla f(o)=0$.    By the identity
$$|\nabla f|^2+R\equiv const.,$$
we see that $R(o)=R_{\max}$ and
\begin{align}\label{identity}
|\nabla f|^2+R=R_{\max}.
\end{align}

Under  (\ref{upper-linear-bound}), the equilibrium point always exists (cf. \cite[Corrollary 2.2]{DZ5}). Moreover,   it is unique since  ${\rm Ric}(g)>0$.  By Morse lemma, we have (cf. \cite{DZ5}).

\begin{lem}\label{sphere-diffeo}
The level set  $\Sigma_r=\{f(x)=r\}$ is  a closed  manifold  for any $r>f(o)$, which is diffeomorphic to $\mathbb{S}^{n-1}$.
\end{lem}

The following lemma is due to \cite[Lemma 3.1]{DZ5}

\begin{lem}\label{geodesic ball in level set}
Let $o\in M$ be an equilibrium point  of   steady  Ricci soliton $(M,g,f)$  with positive Ricci curvature. Then  for   any $p\in M$ and number $k>0$ with $f(p)-\frac{k}{\sqrt{R(p)}}>f(o)$,   it holds
\begin{align}\label{set-mr-contain-1}
B(p,\frac{k}{\sqrt{R_{max}}};  R(p)g)\subset M_{p,k},
\end{align}
where the set $M_{p,k}$ is    defined  by
$$M_{p,k}=\{x\in M| ~ f(p)-\frac{k}{\sqrt{R(p)}}\le f(x)\le f(p)+\frac{k}{\sqrt{R(p)}}\}.$$

\end{lem}

By Lemma \ref{geodesic ball in level set}, we prove

\begin{lem}\label{lem-pointwise curvature estimate}Let $(M,g)$  be  a  steady (gradient) Ricci soliton with nonnegative sectional curvature and positive Ricci curvature. Suppose that (\ref{linear-decay-curvature}) holds.
Then  there exists a constant $C$ such that
\begin{align}\label{pointwise curvature estimate}
\frac{|\nabla R|(p)}{R^{3/2}(p)}\le C,~\forall~p\in ~M.
\end{align}
\end{lem}

\begin{proof}
Fix any $p\in M$ with $f(p)\ge r_0>>1$. Then
\begin{align}
|f(x)-f(p)|\le \frac{1}{\sqrt{R(p)}},~ \forall ~x\in M_{p,1}.\notag
\end{align}
It is known by  \cite{CaCh},
\begin{align}\label{inequality-cao chen}
c_1\rho(x)\le f(x)\le c_2 \rho(x), ~\forall~\rho(x)\ge r_0.
\end{align}
Thus by the curvature decay, we get
\begin{align}
c_2\rho(x)\ge f(p)-\frac{1}{\sqrt{R(p)}}\ge c_1\rho(p)-\sqrt{C_0\rho(p)}.\notag
\end{align}
It follows that
\begin{align}
\frac{R(x)}{R(p)}\le  C_0^2\frac{\rho(p)}{\rho(x)}\le \frac{2c_2C_0^2}{c_1},~\forall x\in M_{p,1}.\notag
\end{align}
On the other hand, by (\ref{set-mr-contain-1}), we have
\begin{align}
B(p,\frac{1}{\sqrt{R_{\max}}};g_p)\subseteq M_{p,1}.\notag
\end{align}
Hence
\begin{align}\label{curvature bound-1}
R(x)\le C^{\prime}R(p),~\forall ~x\in B(p,\frac{1}{\sqrt{R_{\max}}};g_p).
\end{align}

Let $\phi_t$ be generated by $-\nabla f$. Then $g(t)=\phi_t^{\ast}g$ satisfies the Ricci flow,
\begin{align}\label{Ricci flow equation}
         \frac{\partial g(t)}{\partial t} &= -2{\rm Ric}(g(t)).
                  \end{align}
Also rescaled flow  $g_p(t)=R(p) g(R^{-1}(p) t)$  satisfies (\ref{Ricci flow equation}). Since the Ricci curvature is positive,
\begin{align}
B(p,\frac{1}{\sqrt{R_{\max}}};g_p(t))\subseteq B(p,\frac{1}{\sqrt{R_{\max}}};g_p(0)),~t\in~[-1,0].\notag
\end{align}
and
\begin{align}\label{increasing of R}
\frac{\partial}{\partial t}R=2{\rm Ric}(\nabla f,\nabla f)\ge0.
\end{align}
Combining  the above two relations  with  (\ref{curvature bound-1}), we get
\begin{align}
R_{g_p(t)}(x)\le C^{\prime},~\forall~x\in B(p,\frac{1}{\sqrt{R_{\max}}};g_p(0)),~t\in [-1,0].\notag
\end{align}
Thus,  by Shi's higher order estimates \cite{S1}, we obtain
\begin{align}
|\nabla_{(g_p(t))} R_{g_p(t)}|(x)\le C_{1}^{\prime},~\forall~x\in B(p,\frac{1}{ 2\sqrt{R_{\max}}};g_p(-1)),~t\in  [-\frac{1}{2},0].\notag
\end{align}
It follows that
\begin{align}
|\nabla R|(x)\le C_{1}^{\prime}R^{3/2}(p),~\forall~x\in B(p,\frac{1}{2\sqrt{R_{\max}}};g_p(-1)).\notag
\end{align}
In particular, we have
\begin{align}
|\nabla R|(p)\le C_{1}^{\prime}R^{3/2}(p),~{\rm as}~\rho(p)\ge r_0.\notag
\end{align}
The lemma is  proved.
\end{proof}

\begin{rem}\label{rem-pointwise curvature estimate}
From the  argument  in the proof of  Lemma \ref{lem-pointwise curvature estimate},  we can further prove that  there exists a constant $C(k)$  for each $k\in\mathbb{N}$ such that
\begin{align}
\frac{|\nabla^k Rm|(p)}{R^{\frac{k+2}{2}}(p)}\le C(k),~\forall~p\in ~M.\notag
\end{align}
\end{rem}

In Proposition 4.3 of \cite{DZ2},
the authors have obtained a lower decay estimate for steady K\"{a}hler-Ricci solitons with nonnegative bisectional curvature and positive Ricci curvature.   Our proof essentially depends on the Harnack inequality and the existence of equilibrium point\footnote{The existence of equilibrium points  is proved for  steady K\"{a}hler-Ricci solitons  in \cite{DZ1}.}.   Thus  the  argument there still works  for steady Ricci solitons with nonnegative curvature operator and positive Ricci curvature if the soliton admits an equilibrium point. Namely, we have

\begin{theo}\label{lower-bound-scalar-curvature}
Let $(M,g)$ be a $\kappa$-noncollapsed steady Ricci soliton  with nonnegative curvature operator and positive Ricci curvature.  Suppose that $(M,g)$ has an  equilibrium point.  Then   scalar curvature  of $g$ satisfies
(\ref{lower-bound-r}).
\end{theo}

\section{Diameter estimate of level sets}

By Lemma \ref{sphere-diffeo},    there is  a one parameter group of diffeomorphisms  $F_r:$ $~\mathbb{S}^{n-1}\to \Sigma_r\subseteq M$  $(r\ge r_0)$, which is generated by flow
 $$\frac{\partial F_r}{\partial r}=\frac{\nabla f}{|\nabla f|^2}. $$
Let $h_r=F_r^{\ast}(g)$ and $e_i=F_\ast(\bar{e}_i), e_j=F_\ast(\bar{e}_j)$, where $\bar{e}_i,\bar{e}_j\in T\mathbb{S}^{n-1}$.  Then,
\begin{align}\label{diff-flow}
\frac{\partial h_r}{\partial r}(\bar{e}_i,\bar{e}_j)=&{L}_{\frac{\nabla f}{|\nabla f|^2}}g(e_i,e_j)\notag\\
=&\langle \nabla_{e_i}(\frac{\nabla f}{|\nabla f|^2}),e_j\rangle+\langle \nabla_{e_j}(\frac{\nabla f}{|\nabla f|^2}),e_i\rangle\notag\\
=&\langle \frac{\nabla_{e_i}\nabla f}{|\nabla f|^2},e_j\rangle+\langle\nabla f,e_j\rangle\nabla_{e_i}(\frac{1}{|\nabla f|^2})\notag\\
+&\langle \frac{\nabla_{e_j}\nabla f}{|\nabla f|^2},e_i\rangle+\langle\nabla f,e_i\rangle\nabla_{e_j}(\frac{1}{|\nabla f|^2})\notag\\
=&\frac{2}{|\nabla f|^2}{\rm Ric}(e_i,e_j).
\end{align}
 (\ref{diff-flow}) is like  the Ricci flow for metrics   $h_r$ since $|\nabla f|^2(x)$ goes to a constant as $\rho(x)\to\infty$ under
 the condition (\ref{linear-decay-curvature}).  In this section, we use the argument in   \cite[Lemma 8.3 (b)]{Pe} to study the distance functions of $h_r$ in order  to get an  estimate of the diameter of $(\mathbb{S}^{n-1},h_r)$, i.e.,    the diameter of $(\Sigma_r,g)$. First, we prove

\begin{lem}\label{lem-differential inequality of distance function}
Let $(M,g,f)$ be a steady Ricci soliton  as in Theorem \ref{main theorem-2} with positive Ricci curvature. Then  for $x_1,x_2\in\mathbb{S}^{n-1}$
with $d_{r}(x_1,x_2)\ge 2\tau_0$, we have
\begin{align}
\frac{\rm d}{{\rm d}r}d_r(x_1,x_2)\le C(\frac{\tau_0}{r}+\frac{1}{\tau_0}+\frac{d_{r}(x_1,x_2)}{r^{3/2}}),
\end{align}
 where $d_r(\cdot,\cdot)$ is the distance function of $(\mathbb{S}^{n-1},h_r)$.
\end{lem}

\begin{proof}
Let  $\gamma$ be  a normalized minimal geodesic from $x_1$ to $x_2$ with velocity field $X(s)=\frac{{\rm d}\gamma}{{\rm d}s}$ and $V$  any piecewise smooth normal vector field  along $\gamma$  which  vanishes at the endpoints.  By the second variation formula, we have
\begin{align}
\int_{0}^{d_r(x_1,x_2)}(|\nabla_{X}V|^2+\langle \bar R(V,X)V,X\rangle){\rm d}s\ge0.
\end{align}
Let $\{e_i(s)\}_{i=1}^{n-1}$ be a parallel orthonormal frame along $\gamma$ that is perpendicular to $X$. Put $V_i(s)=f(s)e_i(s)$, where $f(s)$ is defined as
\begin{align}
f(s)=\frac{s}{\tau_0},&~\mbox{if}~ 0\le s\le \tau_0;\notag\\
f(s)=1,&~\mbox{if}~\tau_0\le s\le d_r(x_1,x_2)-\tau_0;\notag \\
f(s)=\frac{d_r(x_1,x_2)-s}{\tau_0},&~\mbox{if}~d_r(x_1,x_2)-\tau_0\le s\le d_r(x_1,x_2).\notag
\end{align}
Then $|\nabla_{X}V_i|=|f'(s)|$ and
\begin{align}
\int_{0}^{d_r(x_1,x_2)}|\nabla_XV_i|^2{\rm d}s=2\int_{0}^{\tau_0}\frac{1}{\tau_0^2}{\rm d}s=\frac{2}{\tau_0}.\notag
\end{align}
Moreover
\begin{align}
&\int_{0}^{d_r(x_1,x_2)}\langle \bar R(V_i,X)V_i,X\rangle{\rm d}s\notag\\
&=\int_{0}^{\tau_0}\frac{s^2}{\tau_0^2}\langle \bar R(e_i,X)e_i,X\rangle{\rm d}s
+\int^{d_r(x_1,x_2)-\tau_0}_{\tau_0}\langle \bar R(e_i,X)e_i,X\rangle{\rm d}s\notag\\
&+\int_{d_r(x_1,x_2)-\tau_0}^{d_r(x_1,x_2)}\frac{(d_r(x_1,x_2)-s)^2}{\tau_0^2}\langle \bar R(e_i,X)e_i,X\rangle{\rm d}s.\notag
\end{align}
Thus
\begin{align}\label{second-variation}
0\le &\sum_{i=1}^{n-1}\int_{0}^{d_r(x_1,x_2)}(|\nabla_{X}V_i|^2+\langle \bar R(V_i,X)V_i,X\rangle){\rm d}s\notag\\
=&\frac{2(n-1)}{r_0}-\int_{0}^{d_r(x_1,x_2)}\overline {\rm Ric}(X,X){\rm d}s+\int_{0}^{\tau_0}(1-\frac{s^2}{\tau_0^2})\overline {\rm Ric}(X,X){\rm d}s\notag\\
+&\int^{d_r(x_1,x_2)}_{d_r(x_1,x_2)-\tau_0}(1-\frac{(d_r(x_1,x_2)-s)^2}{\tau_0^2})\overline {\rm Ric}(X,X){\rm d}s.
\end{align}
We claim
\begin{align}\label{curvature-decay-2}
\overline {\rm Ric}(X,X)\le \frac{C}{r}\bar g(X,X),~\forall ~x\in \Sigma_r.
\end{align}

By Gauss formula, we have
\begin{align}
&{\rm Rm}(X,Y,Z,W)=\overline{\rm Rm}(X,Y,Z,W)\notag\\
&
+\langle B(X,Z),B(Y,W)\rangle-\langle B(X,W),B(Y,Z)\rangle,\notag
\end{align}
where $X,Y,Z,W\in T\Sigma_{r}$ and $B(X,Y)=(\nabla_{X}Y)^{\bot}$.  Since
\begin{align*}
B(X,Y)=&\langle \nabla_{X}Y,\nabla f\rangle\cdot\frac{\nabla f}{|\nabla f|^{2}}\\
=&[\nabla_{X}\langle Y,\nabla f\rangle-\langle Y,\nabla_{X}\nabla f\rangle]\cdot\frac{\nabla f}{|\nabla f|^{2}}\\
=&-{\rm Ric}(X,Y)\cdot\frac{\nabla f}{|\nabla f|^{2}},
\end{align*}
we get
\begin{align}\label{gauss-equa}
&{\rm Rm}(X,Y,Z,W)=
\overline{\rm Rm}(X,Y,Z,W)\notag\\
&+ \frac{1}{|\nabla f|^2}({\rm Ric}(X,Z){\rm Ric}(Y,W)- {\rm Ric}(X,W){\rm Ric}(Y,Z))
\end{align}
and
\begin{align}
&R_{ij}=\overline{R}_{ij}+R(\frac{\nabla f}{|\nabla f|},e_{i},e_{j},\frac{\nabla f}{|\nabla f|})-\frac{1}{|\nabla f|^{2}}\sum_k(R_{ij}R_{kk}-R_{ik}R_{kj}),\notag
\end{align}
where  indices  $i,j,k$ are corresponding to  vector fields  on $T\Sigma_{r}$.
Thus for  a unit vector $Y$, we derive
\begin{align}
({\rm Ric}-\overline{\rm Ric})(Y,Y)=&R(\frac{\nabla f}{|\nabla f|},Y,Y,\frac{\nabla f}{|\nabla f|})\notag\\
&-\frac{1}{|\nabla f|^2}\sum_{i=1}^{n-1}[{\rm Ric}(Y,Y){\rm Ric}(e_i,e_i)-{\rm Ric}^2(Y,e_i)].\notag
\end{align}
Note that
\begin{align}
R(\frac{\nabla f}{|\nabla f|},Y,Y,\frac{\nabla f}{|\nabla f|})\le& {\rm Ric}(\frac{\nabla f}{|\nabla f|},\frac{\nabla f}{|\nabla f|})
=\frac{|\langle \nabla R,\nabla f\rangle|}{|\nabla f|^2}\notag\\
\le &\frac{|\nabla R|}{|\nabla f|}
\le \frac{C}{r^{3/2}},\notag
\end{align}
and
\begin{align}
\frac{1}{|\nabla f|^2}|\sum_{i=1}^{n-1}[{\rm Ric}(Y,Y){\rm Ric}(e_i,e_i)-{\rm Ric}^2(Y,e_i)]|
\le \frac{R^2+|{\rm Ric}|^2}{|\nabla f|^2}\le \frac{C}{r^2}.\notag
\end{align}
Hence, we obtain
\begin{align}\label{diference-ricci}
| ({\rm Ric}-\overline{\rm Ric})(Y,Y) | \le  \frac{C}{r^{3/2}}, ~r\ge ~r_0.
\end{align}
In particular,
\begin{align}
 \overline {\rm Ric}(Y,Y)\le \frac{C_1}{ r}, ~r\ge ~r_0.\notag
\end{align}
This proves (\ref{curvature-decay-2}).

By  (\ref{second-variation}) and (\ref{curvature-decay-2}),  it is easy to see
\begin{align}\label{integral-bar-ricci}
\int_{0}^{d_r(x_1,x_2)}\overline {\rm Ric}(Y,Y){\rm d}s\le\frac{2(n-1)}{\tau_0}+\frac{4C\tau_0}{3r}.
\end{align}
Also by (\ref{diference-ricci}), we see
\begin{align}\label{ric-differen}\int_{0}^{d_r(x_1,x_2)}({\rm Ric}-\overline{\rm Ric})(Y,Y){\rm d}s\le \frac{C}{r^{3/2}}d_r(x_1,x_2).
\end{align}
On the other hand, if we
let $Y(s)=(F_r)_{\ast}(X(s))$ with $|Y(s)|_{(\Sigma_r,g)}\equiv 1$,  then by  (\ref{diff-flow}),  we have
\begin{align}
&\frac{{\rm d}}{{\rm d}r}d_r(x_1,x_2)\notag\\=&\frac{1}{2}\int_{0}^{d_r(x_1,x_2)}(L_{\frac{\nabla f}{|\nabla f|^2}}g)(Y,Y){\rm d}s\notag\\
=&\frac{1}{|\nabla f|^2}\int_{0}^{d_r(x_1,x_2)}\overline{\rm Ric}(Y,Y){\rm d}s+\frac{1}{|\nabla f|^2}\int_{0}^{d_r(x_1,x_2)}({\rm Ric}-\overline{\rm Ric})(Y,Y){\rm d}s\notag.
\end{align}
Thus inserting  (\ref{integral-bar-ricci}) and (\ref{ric-differen}) into the above relation,
we obtain
$$\frac{{\rm d}}{{\rm d}r}d_r(x_1,x_2)\le C(\frac{\tau_0}{r}+\frac{1}{\tau_0}+\frac{d_{r}(x_1,x_2)}{r^{3/2}}).$$

\end{proof}

\begin{cor}\label{cor-differential inequality of distance function}  Let $d_r(\cdot,\cdot)$ be the  distance function as in Lemma \ref{lem-differential inequality of distance function}.
Then  for any $x_1,x_2\in\mathbb{S}^{n-1}$, we have
\begin{align}
\frac{\rm d}{{\rm d}r}d_r(x_1,x_2)\le C(\frac{2}{\sqrt{r}}+\frac{d_r(x_1,x_2)}{r^{3/2}}),~r\ge r_0.
\end{align}
\end{cor}

\begin{proof}
If $d_r(x_1,x_2)\ge 2\sqrt{r}$,    the corollary follows from Lemma \ref{lem-differential inequality of distance function} by taking $\tau_0=\sqrt{r}$. If $d_r(x_1,x_2)< 2\sqrt{r}$,    we have
\begin{align}
\frac{{\rm d}}{{\rm d}r}d_r(x_1,x_2)=&\frac{1}{2}\int_{0}^{d_r(x_1,x_2)}(L_{\frac{\nabla f}{|\nabla f|^2}}g)(Y,Y){\rm d}s\notag\\
=&\frac{1}{|\nabla f|^2}\int_{0}^{d_r(x_1,x_2)} {\rm Ric}(Y,Y){\rm d}s\notag\\
\le&\frac{C}{r}d_r(x_1,x_2)
\le\frac{2C}{\sqrt{r}}.\notag
\end{align}
The corollary is proved.
\end{proof}

By Corollary \ref{cor-differential inequality of distance function},  we get the following  diameter estimate for  $(\Sigma_r,\bar g)$.

\begin{prop}\label{theorem-diameter estimate}Let $(M,g,f)$ be a steady Ricci soliton  as in Theorem \ref{main theorem-2} with positive Ricci curvature. Then
there exists a constant $C$ independent of $r$ such that
\begin{align}
{\rm diam}(\Sigma_r,g)\le C\sqrt{r},~\forall~r\ge r_0.
\end{align}

\end{prop}

\begin{proof}
For any fixed $x_1,x_2\in \mathbb{S}^{n-1}$, by Corollary \ref{cor-differential inequality of distance function}, we have
\begin{align}
d_r(x_1,x_2)\le &e^{\int_{r_0}^r\tau^{-3/2}{\rm d}\tau}(d_{r_0}(x_1,x_2)+\int_{r_0}^r\frac{2C}{\sqrt{\tau}}e^{-\int_{r_0}^{\tau}s^{-3/2}{\rm d}s}{\rm d}\tau)\notag\\
\le& 4C(\sqrt{r}-\sqrt{r_0})+\frac{2}{\sqrt{r_0}}\cdot {\rm diam}(\mathbb{S}^{n-1},h_{r_0}).\notag
\end{align}
Thus
\begin{align}
{\rm diam}(\Sigma_r,g)={\rm diam}(\mathbb{S}^{n-1},h_{r})\le 4C'\sqrt{r},~r\ge r_0.\notag
\end{align}

\end{proof}

As a corollary of Proposition \ref{theorem-diameter estimate},  we get

\begin{cor}\label{set-mr-contain-2} Let $(M,g,f)$ be a steady Ricci soliton  as in Theorem \ref{main theorem-2} with positive Ricci curvature. Then
 there exists  a uniform constant $C_0>0$ such that the following is true:
  for any $k\in\mathbb{N}$, there exists $\bar r_0=\bar r_0(k)$ such that
\begin{align}
M_{p,k}\subset B(p,C_0+\frac{2k}{\sqrt{R_{\max}}} ; R(p)g), ~\forall ~ \rho(p)\ge  \bar r_0.
\end{align}
\end{cor}

\begin{proof}
By Proposition  \ref{theorem-diameter estimate} and  (\ref{linear-decay-curvature}),
 it is easy to see that
\begin{align}\label{mr-set}
\Sigma_{f(p)}\subset B(p,C_0;  R(p)g)
\end{align}
for some uniform  constant $C_0$ as long as  $f(p)$ is large enough.
 Note that by (\ref{identity}),
\begin{align}
\frac{R_{\max}}{2}\le |\nabla f|^2(x)\le R_{\max},~\forall~x\in M_{p,k}, ~\rho(p)\ge r_0.\notag
\end{align}
Since  there exists  $q^{\prime}\in \Sigma_{f(p)}$  for  any $q\in M_{p,k}$ such that $\phi_{s}(q)=q^{\prime}$ for some $s\in \mathbb{R}$,  we derive
\begin{align*}
d(p,q)\leq& d(p, q^{\prime})  + d(q^{\prime}, q)\\
\leq& {\rm diam}(\Sigma_{f(p)},g)+\mathcal{L}(\phi_{\tau}|_{[0,s]})\\
\leq& C_0R^{-\frac{1}{2}}(p)+|\int_{0}^{s}|\frac{d\phi_{\tau}(q)}{d\tau}|d\tau|\\
=& C_0R^{-\frac{1}{2}}(p)+\int_{0}^{s}|\nabla f(\phi_{\tau}(q))|d\tau\\
\le& C_0R^{-\frac{1}{2}}(p)+\int_{0}^{s}|\nabla f(\phi_{\tau}(q))|^{2}\cdot \frac{2}{\sqrt{R_{\max}}}d\tau\\
=& C_0R^{-\frac{1}{2}}(p)+|\int_{0}^{s}\frac{d(f(\phi_{\tau}(q)))}{d\tau}\cdot \frac{2}{\sqrt{R_{\max}}}d\tau|
\end{align*}
\begin{align*}
\leq& C_0R^{-\frac{1}{2}}(p)+|f(q)-f(p)|\cdot \frac{2}{\sqrt{R_{\max}}}\\
\leq& \Big(C_0+\frac{2k}{\sqrt{R_{\max}}}\Big)\cdot \frac{1}{\sqrt{R(p)}}.
\end{align*}
This implies
\begin{align*}
M_{p,k}\subset B(p,C_0+\frac{2k}{\sqrt{R_{\max}}} ; R(p)g).
\end{align*}

\end{proof}

Corollary  \ref{set-mr-contain-2}  will be used in the next Section.

\section{Proof of Theorem \ref{main theorem-2}-I }

In this section, we prove Theorem \ref{main theorem-2}  for steady Ricci solitons with positive Ricci curvature as follows.

\begin{theo}\label{main theorem-strong}
Let $(M,g,f)$ be a   $\kappa$-noncollapsed steady Ricci soliton with nonnegative sectional curvature and positive Ricci curvature.
Suppose that (\ref{linear-decay-curvature}) is satisfied.  Then for any $p_{i}\rightarrow+\infty$,
 rescaled flows $(M, R(p_i)g(R^{-1}(p_i)t),$ $p_{i})$ converge subsequently to
$(\mathbb{R}\times \mathbb{S}^{n-1},ds^2+g_{\mathbb{S}^{n-1}}(t))$ ( $t\in (-\infty,0]$) in the Cheeger-Gromov topology, where
 $(\mathbb{S}^{n-1},$ $g_{\mathbb{S}^{n-1}}(t))$ is a $\kappa$-noncollapsed ancient Ricci flow with nonnegative sectional curvature. Moreover,  scalar curvature of $g_{\mathbb{S}^{n-1}}(t))$ satisfies (\ref{curvature-decay-sn}).
\end{theo}

We need several lemmas to prepare for the proof of  Theorem \ref{main theorem-strong}.   First  we give a volume comparison for  level sets $\Sigma_r$.

\begin{lem}
Under the condition of Theorem \ref{main theorem-strong}, for any small $\varepsilon>0$,  there is a $r_0>0$ such that for any $s\in[-1,1]$,
\begin{align}\label{volume-comparison-sigmar}
1-\varepsilon\le \frac{{\rm vol}(\Sigma_{r+s\sqrt{r}},g)}{{\rm vol}(\Sigma_r,g)}\le 1+\varepsilon, ~ r\ge ~r_0.
\end{align}
\end{lem}

\begin{proof}
Let  $V\in T\mathbb{S}^{n-1}$ be  any fixed nonzero vector.  By (\ref{diff-flow}), we have
\begin{align}
|\frac{\partial}{\partial r}h_r(V,V)|=&\frac{2}{|\nabla f|^2}|{\rm Ric}((F_r)_{\ast}V,(F_r)_{\ast}V)|
\le\frac{C}{r}h_r(V,V).\notag
\end{align}
Thus
\begin{align}
\frac{r_1^C}{r_2^C}\le \frac{h_{r_2}(V,V)}{h_{r_1}(V,V)}\le \frac{r_2^C}{r_1^C},~\forall~0\le r_1\le r_2.\notag
\end{align}
It follows
\begin{align}
\frac{h_{r+s\sqrt{r}}(V,V)}{h_{r}(V,V)}\to 1,~\mbox{as}~r\to\infty\notag
\end{align}
and
\begin{align}
\frac{\det (h_{r+s\sqrt{r}})}{\det(h_{r})}\to 1,~\mbox{as}~r\to\infty.\notag
\end{align}
Hence (\ref{volume-comparison-sigmar}) follows.
\end{proof}

Let $ g_p=R(p)g$ be a rescaled metric of $(M,g)$. Then

\begin{lem}\label{lem-volume lower bound}
Under the condition of Theorem \ref{main theorem-strong}, there exists a constant $C(\kappa)$ such that
\begin{align}\label{volume-lower-bound}
{\rm vol}(\Sigma_{f(p)},{g}_p)\ge C(\kappa), ~{\rm if}f(p)\ge r_0.
\end{align}

\end{lem}

\begin{proof}
We  define a set $M_{r}(s)$ by
\begin{align}
M_{r}(s)=\{x\in M|r-s\sqrt{r}\le f(x)\le r+s\sqrt{r}\}.
\end{align}
By the uniform curvature decay of $R(x)$, it is easy to see that there is a  constant $c>0$ such that
\begin{align}
M_{p,cs}\subseteq M_{f(p)}(s), ~\forall~s\in [-1,1],
\end{align}
as long as  $f(p)$ is large enough. Then by Lemma \ref{geodesic ball in level set}, we see
\begin{align}
B(p,R_{\max}^{-\frac{1}{2}}; g_p)\subseteq  M_{p,1}\subseteq M_{f(p)}(c^{-1}).\notag
\end{align}
Note that
\begin{align}
R_{ g_p}(x)=\frac{R(x)}{R(p)}\le C,~\forall~x\in M_{f(p)}(c^{-1}).\notag
\end{align}
Thus
\begin{align}
R_{g_p}(x)\le C, ~{\rm in}~B(p,R_{\max}^{-\frac{1}{2}}; g_p).\notag
\end{align}
Since $(M, g_p)$ is $\kappa$-noncollapsed, we get
\begin{align}
{\rm vol}( M_{f(p)}(c^{-1}), g_p)\ge {\rm vol}(B(p,R_{\max}^{-\frac{1}{2}}; g_p))\ge c(\kappa).\notag
\end{align}
On the other hand, by the co-area formula and (\ref{volume-comparison-sigmar}), we have
\begin{align}
{\rm vol}( M_{f(p)}(c^{-1}),g_p)=&R^{\frac{n}{2}}(p){\rm vol}( M_{f(p)}(c^{-1}),g)\notag\\
=&R^{\frac{n}{2}}(p)\int_{f(p)-c^{-1}\sqrt{f(p)}}^{f(p)+c^{-1}\sqrt{f(p)}}   \frac{1}{|\nabla f|}{\rm vol}(\mathbb{S}^{n-1},h_r){\rm d}r\notag\\
\le& R^{\frac{n}{2}}(p)\frac{ 4}{c\sqrt{R_{\max}}} \sqrt{f(p)} {\rm vol}(\mathbb{S}^{n-1},h_{f(p)})\notag\\
\le&C\cdot{\rm vol}(\Sigma_{f(p)}, {g}_p)\notag.
\end{align}
Hence, (\ref{volume-lower-bound}) follows from  the above inequalities.
\end{proof}

\begin{lem}\label{lem-tensor computation}
 Let $\bar{D}\in \otimes_{i=1}^{k}T^{\ast}\Sigma_r$ be a $k$-multiple tensor on $\Sigma_r$ and $D\in \otimes_{i=1}^{k}T^{\ast}M$  a $k$-multiple tensor on $M$, respectively. Then, under the condition of Theorem \ref{main theorem-strong}, we have
\begin{align}\label{connection-relation}
&(\bar{\nabla}\bar{D})(e_{i_0},e_{i_1},\cdots,e_{i_k})=(\nabla D)(e_{i_0},e_{i_1},\cdots,e_{i_k})\notag\\
&+\sum_{s=1}^{k}D(e_{i_1},\cdots,e_{i_{s-1}},\frac{\nabla f}{|\nabla f|^2},e_{i_{s+1}},\cdots,e_{i_k}){\rm Ric}(e_{i_0},e_{i_s}).
\end{align}
\end{lem}

\begin{proof} Let  $e_{i_1},e_{i_2},\cdots,e_{i_k}$ be unit vector fields which are tangent to $\Sigma_r$. Let
 $\nabla$ and $\bar{\nabla}$ be the Levi-Civita connections $M$ and $\Sigma_r$, respectively. Then
\begin{align}
(\bar{\nabla}\bar{D})(e_{i_0},\cdots,e_{i_k})&=e_{i_0}[\bar{D}(e_{i_1},\cdots,e_{i_k})]\notag\\
&+\sum_{s=1}^{k}\bar{D}(e_{i_1},\cdots,e_{i_{s-1}},\bar{\nabla}_{e_{i_0}}e_{i_s},e_{i_{s+1}},\cdots,e_{i_k}).\notag
\end{align}
and
\begin{align}
(\nabla D)(e_{i_0},\cdots,e_{i_k})&= e_{i_0}[D(e_{i_1},\cdots,e_{i_k})]\notag\\
&+\sum_{s=1}^{k}D(e_{i_1},\cdots,e_{i_{s-1}},\nabla_{e_{i_0}}e_{i_s},e_{i_{s+1}},\cdots,e_{i_k}).\notag
\end{align}
Note
\begin{align}
\nabla_{e_{i_0}}e_{i_s}-\bar{\nabla}_{e_{i_0}}e_{i_s}&=\langle\nabla_{e_{i_0}}e_{i_s},\frac{\nabla f}{|\nabla f|}\rangle \frac{\nabla f}{|\nabla f|}\notag\\
&=-\langle e_{i_s},\nabla_{e_{i_0}}\nabla f\rangle \frac{\nabla f}{|\nabla f|^2}\notag\\
&=-{\rm Ric}(e_{i_0},e_{i_s})\frac{\nabla f}{|\nabla f|^2}.\notag
\end{align}
Combining the identities above, we get (\ref{connection-relation}).
\end{proof}

\begin{prop}\label{cor-limit of level set}
Under the condition of Theorem \ref{main theorem-strong}, for any $p_i\to\infty$, $(\Sigma_{f(p_i)},\overline{g}_{p_i},p_i)$ converges subsequently to $(S_{\infty},h_{\infty},p_{\infty})$ in Cheeger-Gromov sense as $i\to\infty$. Here $\bar g_{p_i}=R(p_i)\bar g$ and $(\Sigma_{f(p_i)},\bar g)$ is as a hypersurface  of $(M,g)$ with induced metric $\bar g$. Moreover, $S_{\infty}$ is diffeomorphic to $\mathbb{S}^{n-1}$.
\end{prop}

\begin{proof}
By (\ref{gauss-equa}), we have
\begin{align}
&{\rm  Rm}(X,Y,Z,W)=
 \overline {\rm Rm}(X,Y,Z,W)\notag\\
&+\frac{1}{R_{\max}-R}({\rm Ric}(X,Z){\rm Ric}(Y,W)- {\rm Ric}(X,W){\rm Ric}(Y,Z)).\notag
\end{align}
Let
$$D^{(0)}={\rm Rm}-\frac{1}{R_{\max}-R}{\rm Ric}\wedge{\rm Ric}$$
be a $(0,4)$-tpye tensor on $M$. Then $D^{(0)}|_{\Sigma_{f(p_i)}}= \overline {\rm Rm}$.
Note that by Remark \ref{rem-pointwise curvature estimate} we have
\begin{align}
\frac{|\nabla^k {\rm Rm}|(x)}{R^{\frac{k+2}{2}}(p_i)}=\frac{|\nabla^k {\rm Rm}|(x)}{R^{\frac{k+2}{2}}(x)}\cdot\frac{R^{\frac{k+2}{2}}(x)}{R^{\frac{k+2}{2}}(p_i)}\le C(k),~\forall ~ x\in \Sigma_{f(p_i)}.\notag
\end{align}
Since
\begin{align}
\nabla(\frac{1}{R_{\max}-R})=\frac{\nabla R}{(R-R_{\max})^2},\notag
\end{align}
by  induction on $m$, we get
\begin{align}\label{rm-nabla}
|\nabla^m   D^{(0)} |(x)\le C(m)R^{\frac{m+2}{2}}(p_i),~\forall~x\in \Sigma_{f(p_i)}.
\end{align}

Let
\begin{align}
&D^{(k)}=\nabla D^{(k-1)}\notag\\
&+\sum_{s=1}^{k+4}D^{(k-1)}(e_{i_1},\cdots,e_{i_{s-1}},\frac{\nabla f}{|\nabla f|^2},e_{i_{s+1}},\cdots,e_{i_{k+4}}){\rm Ric}(e_{i_0},e_{i_s}).\notag
\end{align}
Then by Lemma \ref{lem-tensor computation}, we have
\begin{align}
\bar{\nabla}^{k}\overline{\rm Rm}=D^{(k)}|_{\Sigma_{f(p_i)}}.\notag
\end{align}
On the other hand,  by induction on $k$ with the help of (\ref{rm-nabla}), we get
\begin{align}
|\nabla^mD^{(k)}|(x)\le C(m,k)R^{\frac{m+k+2}{2}}(p_i),~\forall~x\in \Sigma_{f(p_i)}.\notag
\end{align}
In particular,
\begin{align}
|D^{(k)}|(x)\le C(k)R^{\frac{k+2}{2}}(p_i),~\forall~x\in \Sigma_{f(p_i)}.\notag
\end{align}
Thus
\begin{align}\label{higher-derivative-curvature}
|\bar \nabla^k_{\bar g_{p_i}}{\rm \overline{Rm}} |_{\bar g_{p_i}}(x)\le \frac{|D^{(k)}|(x)}{R^{\frac{k+2}{2}}(p_i)}\le C(k),~\forall ~ x\in \Sigma_{f(p_i)},
\end{align}

By Lemma \ref{lem-volume lower bound} and  Theorem \ref{theorem-diameter estimate}, respectively,  we have
\begin{align}
{\rm vol}(\Sigma_{f(p_i)},\bar {g}_{p_i})\ge C(\kappa)\notag
\end{align}
  and
\begin{align}
{\rm diam}(\Sigma_{f(p_i)},\bar {g}_{p_i})\le C.\notag
\end{align}
Then  by Cheeger-Gromov compactness theorem together with (\ref{higher-derivative-curvature}), we see that  $(\Sigma_{f(p_i)},\overline{g}_{p_i},p_i)$ converges subsequently to $(S_{\infty},h_{\infty},p_{\infty})$.  Note  that $\Sigma_{f(p_i)}$  are all diffeomorphic to $\mathbb{S}^{n-1}$.  Therefore, $S_{\infty}$ is  also diffeomorphic to $\mathbb{S}^{n-1}$.

\end{proof}

\subsection{Proof of Theorem \ref{main theorem-strong}}
We are now in a  position to prove Theorem \ref{main theorem-strong}. The proof consists of the following three lemmas.   First, by the arguments in \cite{DZ2, DZ5},  we prove

\begin{lem}\label{lem-convergence and splitting}
Under the condition of Theorem \ref{main theorem-strong}, let $p_i\to\infty$.   Then  by taking a subsequence of $p_i$ if necessary,   we have
\begin{align*}
(M,g_{p_{i}}(t),p_{i})\rightarrow(\mathbb{R}\times N,g_{\infty}(t); p_{\infty}),~for~t\in(-\infty,0],
\end{align*}
where $g_{p_i}(t)=R(p_{i})g(R^{-1}(p_{i})t)$, $g_{\infty}(t)=ds\otimes ds+g_{ N}(t)$ and $( N, g_{N}(t))$ is an ancient solution of Ricci flow on  $N$.
\end{lem}

\begin{proof}
 Fix $\overline{r}>0$.    By (\ref{linear-decay-curvature}),   it is easy to see  that there exists  a uniform $C_1$  independent of $\overline{r}$  such that
\begin{align}\label{scalar-estimate-2}
R(x)\le C_1R(p_i),~\forall~x\in M_{p_i,\overline{r}\sqrt{R_{\max}}}
\end{align}
as long as $i$ is large enough.  Then by Lemma \ref{set-mr-contain-1},  we have
\begin{align}
R_{g_{p_i}}(x)\le C_1,~\forall~x\in B(p_i,\overline{r};g_{p_i}),\notag
\end{align}
where $g_{p_i}=g_{p_i}(0)$.   Since the scalar curvature is increasing along the flow (cf. (\ref{increasing of R})) and the sectional curvature is nonnegative,  for any $t\in(-\infty,0]$,  we get
\begin{align}
|{\rm Rm}_{g_{p_i}(t)}(x)|_{g_{p_i}(t)}&\le C(n)R_{g_{p_i}(t)}(x)\notag\\
&\le C(n)R_{g_{p_i}}(x)\le C(n)C_1,~\forall~x\in B(p_i,\overline{r};g_{p_i}).\notag
\end{align}
Note that $(M,g(t))$ is $\kappa$-noncollapsed.   Hence $g_{p_i}(t)$ converges subsequently  to a limit flow $(M_{\infty}, g_{\infty}(t); p_\infty)$ for  $t\in(-\infty,0]$ \cite{H1}.
Moreover,  the limit flow  has uniformly bounded curvature.  It remains  to prove the splitting property.

 Let $X_{(i)}=R(p_{i})^{-\frac{1}{2}}\nabla f$. Then
\begin{align}
\sup_{ B(p_{i},\bar r ;  {g_{p_i}})}| \nabla_{(g_{p_i})}X_{(i)}|_{g_{p_i}}&= \sup_{ B(p_{i},\bar r ;  {g_{p_i}})}\frac{|{\rm Ric}|}{\sqrt{R(p_{i})}}\notag\\
&\le C\sqrt{R(p_{i})} \to 0.\notag
\end{align}
By Remark \ref{rem-pointwise curvature estimate}, it follows that
$$\sup_{ B(p_{i},\bar r;  {g_{p_i}})}| \nabla^{m}_{(g_{p_i})}X_{(i)}|_{g_{p_i}}\leq C(n)\sup_{ B(p_{i},\bar r;  {g_{p_i}})}| \nabla^{m-1}_{(g_{p_i})}{\rm Ric}({g_{p_i})}|_{g_{p_i}}\le C_1.$$
Thus  $X_{(i)}$ converges  subsequently  to a parallel  vector field $X_{(\infty)}$ on $( M_{\infty},$ $  g_{\infty}(0))$.
 Moreover,
 \begin{align}
 |X_{(i)}|_{g_{p_i}}( x)=|\nabla f|(p_{i})
=\sqrt{R_{\rm max}}+o(1)>0, ~\forall~ x\in B(p_{i},\bar r ;  {g_{i}}),\notag
 \end{align}
as long as $f(p_i)$ is large enough.  This implies that $X_{(\infty)}$ is non-trivial.
 Hence,  $( M_{\infty},g_{\infty}(t))$ locally splits off a piece of  line along $X_{(\infty)}$. In the following, we  show that  $X_{(\infty)}$
 generates a line through $p_\infty$.

 By Corollary \ref{set-mr-contain-2},
\begin{align}
M_{p_i,k}\subset B(p_i,C_0+\frac{2k}{\sqrt{R_{\max}}} ; g_{p_i}(0)), ~\forall ~ p_i\to\infty.\notag
\end{align}
Let $\gamma_{i,k}(s)$, $s\in(-D_{i,k},E_{i,k})$ be an integral curve generated by $X_{(i)}$ through $p_i$, which  restricted in $M_{p,k}$, i.e., $\gamma_{i,k}(s)$ satisfies $f(\gamma_{i,k}(-D_{i,k}))=f(p_i)-\frac{k}{\sqrt{R(p_i)}}$ and $f(\gamma_{i,k}(E_{i,k}))=f(p_i)+\frac{k}{\sqrt{R(p_i)}}$. Then $\gamma_{i,k}(s)$ converges to a  geodesic $\gamma_\infty(s)$ generated by $X_{(\infty)}$ through $p_\infty$,  which restricted in
$B(p_\infty,$ $2\pi\sqrt{B}+\frac{2k}{\sqrt{R_{\max}}};g_\infty(0))$.  If we let  $L_{i,k}$ be   lengths  of $\gamma_{i,k}(s)$
and $L_{\infty,k}$  length  of    $\gamma_\infty(s)$, respectively,
\begin{align}
L_{i,k}=&\int_{-D_{i,k}}^{E_{i,k}}|X_{(i)}|_{g_{p_i}(0)} ds\notag\\
=&\int_{f(p_i)-\frac{k}{\sqrt{R(p_i)}}}^{f(p_i)+\frac{k}{\sqrt{R(p_i)}}} \frac{|X_{(i)}|_{g_{p_i}(0)}}{\langle\nabla f,X_{(i)}\rangle}df   \notag\\
=&  \int_{f(p_i)-\frac{k}{\sqrt{R(p_i)}}}^{f(p_i)+\frac{k}  {\sqrt{R(p_i)}}}  \frac{\sqrt{R(p_i)} }{|\nabla f|_{g}} df\notag\\
\ge& 2 R_{max}^{-\frac{1}{2}}k, \notag
\end{align}
and so,
\begin{align}
L_{\infty,k}\ge \frac{1}{2}L_{i,k}\ge  R_{max}^{-\frac{1}{2}} k. \notag
\end{align}
 Thus  $X_{(\infty)}$ generates a line   $\gamma_\infty(s)$ through $p_\infty$ as $k\to \infty$. As a consequence,  $( M_{\infty},g_{\infty}(0))$ splits off a  line
 and so does the flow $(M_{\infty},g_{\infty}(t); p_{\infty})$. The lemma is proved.

\end{proof}

Next, we show

\begin{lem}\label{sphere-diff}
 $N$ is diffeomorphic to $\mathbb{S}^{n-1}$.
\end{lem}

\begin{proof}
We claim that $N$ is connected.  In fact, for all $k\in\mathbb{N}$, by the convergence result in Lemma \ref{lem-convergence and splitting},  there are diffeomorphisms $\Phi_{i_k}:U_{i_k}(\subseteq M_{\infty})\to B(p_{i_k},k;R(p_{i_k})g)$ with $\Phi_{i_k}(p_{\infty})=p_{i_k}$. Moreover, $(U,\Phi_{i_k}^{\ast}(R(p_{i_k})g))$ is $C^m$ close to $(U,g_{\infty}(0))$ for any $m\in \mathbb{N}$. Thus
\begin{align}
(\Phi_{i_k}) \big(B(p_{\infty},\frac{k}{2};g_{\infty}(0))\big)\subseteq B(p_{i_k},k;R(p_{i_k})g),\notag
\end{align}
 and$(\Phi_{i_k}^{-1}) \big(B(p_{i_k},k; R(p_{i_k})g)\big)$ exhausts  $M_{\infty}$ as $k\to \infty$. For any $q\in B(p_{i_k},k;$ $R(p_{i_k})g)$, there exists a minimal geodesic $\gamma(s):[0,l]\to M$ such that $\gamma(0)=p_{i_k}$, $\gamma(l)=q$. Note  $\gamma|_{[0,l]}\subseteq B(p_{i_k},k;R(p_{i_k})g)$. It follows that $(\Phi_{i_k}^{-1}) \big(B(p_{i_k},k;R(p_{i_k})g)\big)$ is connected for each $k$. Therefore, $M_{\infty}$ is connected,  and so is $N$.

By Proposition  \ref{cor-limit of level set} and Lemma \ref{lem-convergence and splitting}, we may assume that $(\Sigma_{f(p_i)},g_{p_i},p_i)$ converge to a limit $(\mathbb{S}^{n-1},h_{\infty},p_{\infty})$ and $\mathbb{S}^{n-1}\subseteq M_{\infty}=N\times\mathbb{R}$. Then  by the above claim, it suffices to prove that $\mathbb{S}^{n-1}\subseteq N\times\{p_{\infty}\}$.  Let  $X_{(\infty)}$  and $X_{i}$ be  the vector fields defined  as in  the proof of Lemma \ref{lem-convergence and splitting}. Let  $V\in T\mathbb{S}^{n-1}$  with $|V|_{g_\infty}=1.$
Thus by the convergence in Proposition  \ref{cor-limit of level set},  we see that there are  $V_{(i)}\in T\Sigma_{f(p_i)}$ such that $R(p_i)^{-\frac{1}{2}}V_{(i)}\to V$. It follows
\begin{align}
g_{\infty}(V,X_{(\infty)})=\lim_{i\to\infty}R(p_i)g(  R(p_i)^{-\frac{1}{2}}V_{(i)},X_{(i)})=\lim_{i\to\infty}g(  V_{(i)}, \nabla f)=0.\notag
\end{align}
This shows that  $V$ is vertical to $X_{(\infty)}$ for any $V\in T\mathbb{S}^{n-1}$. Hence, $\mathbb{S}^{n-1}\subseteq N\times\{p_{\infty}\}$. Note that ${\rm dim} N=n-1$. We complete the proof.

\end{proof}

Finally, we verify the condition (\ref{curvature-decay-sn}). We prove

\begin{lem}\label{lem-decay of t}
Let $(M_{\infty}=N\times \mathbb R,g_{\infty}(t))$ be the limit manifold in Lemma \ref{lem-convergence and splitting}. Then,  scalar curvature $R_{\infty}(x,t)$  of  $g_{\infty}(t)$  satisfies
\begin{align}\label{decay-on-sigma}
R_{\infty}(x,t)\le \frac{C}{|t|},~\forall~t<0,~x\in M_{\infty}.
\end{align}
\end{lem}

\begin{proof}
Let $\phi_{t}$ be generated by $-\nabla f$. Then,
\begin{align}
\frac{{\rm d}f(\phi_{t}(p))}{{\rm d}t}=-|\nabla f|^2(\phi_{t}(p))\notag
\end{align}
and
\begin{align}
\frac{{\rm d}|\nabla f|^2(\phi_{t}(p))}{{\rm d}t}=-2{\rm Ric}(\nabla f,\nabla f)(\phi_{t}(p))\le 0.\notag
\end{align}
It follows that
\begin{align}
|\nabla f|^2(\phi_{\tau}(p))\le |\nabla f|^2(\phi_{t}(p)),~\forall~\tau\ge t.\notag
\end{align}
Hence,
\begin{align}
f(\phi_{t}(p))\ge f(\phi_{\tau}(p))+|t-\tau||\nabla f|^2(\phi_{\tau}(p)).\notag
\end{align}
By taking $\tau=0$,   for any $p\in\{q\in M|f(q)\ge 1 \}$,  we get
\begin{align}
f(\phi_{t}(p))\ge f(p)+|t||\nabla f|^2(p)\ge 1+c|t|, \notag
\end{align}
where $c=\min_{p\in \Sigma_1}|\nabla f|^2$.
On the other hand,  by  (\ref{linear-decay-curvature}) and (\ref{inequality-cao chen}), we have
\begin{align} R(p)\le \frac{C}{f(p)},~\forall~p\in \{q\in M|~ f(q)\ge 1 \}.\notag
\end{align}
Hence,
\begin{align}\label{lem-uniform decay of t on M}
R(p,t)\le \frac{C}{f(\phi_{t}(p))}\le \frac{C}{1+c|t|},~\forall~p\in\{q\in M|~f(q)\ge 1 \}.
\end{align}

Let $x\in M_{\infty}$ and  $d_{g_{\infty}(0)}(x,p_{\infty})=\overline{r}$.  By the convergence in Lemma \ref{lem-convergence and splitting}, there are $x_i\in B(p_i,2\overline{r};g_{p_i}(0))$ such that $x_i\to x$ as $i\to\infty$.  Moreover,
\begin{align}
\lim_{i\to\infty} R_{g_{p_i}(t)}(x_i)=R_{\infty}(x,t).\notag
\end{align}
Note that $x_i\in B(p_i,2\overline{r};g_{p_i}(0))\subseteq M_{p_i,2\overline{r}\sqrt{R_{\max}}}$. It means that
\begin{align}
f(x_i)\ge f(p_i)-2\overline{r}\sqrt{\frac{R_{\max}}{R(p_i)}}>>1,~as~i\to\infty.\notag
\end{align}
Hence, by (\ref{lem-uniform decay of t on M}), we derive
\begin{align}
R_{g_{p_i}(t)}(x_i)=\frac{R(x_i,R^{-1}(p_i)t)}{R(p_i)}\le \frac{C}{R(p_i)+c|t|}\le \frac{2C}{c|t|},~as~i\to\infty.\notag
\end{align}
Let $i\to\infty$, we get  (\ref{decay-on-sigma}).
\end{proof}

 By Lemma \ref{lem-convergence and splitting} and Lemma \ref{sphere-diff},   (\ref{decay-on-sigma}) in Lemma \ref{lem-decay of t} implies (\ref{curvature-decay-sn}). The  proof of Theorem  \ref{main theorem-strong}  is completed.

\section{Proof of Theorem \ref{main theorem-2}-II}

In this section,  we complete the proof of  Theorem \ref{main theorem-2}.    We need to  describe  the  structure  of level set  $\Sigma_r$ of $(M,g,f)$ without assumption of  positive Ricci curvature.

\begin{lem}\label{lem-level set structure nonnegative case}
Let $(M,g,f)$ be a   non-flat steady Ricci soliton with nonnegative sectional curvature.  Let $S=\{p\in M|\nabla f(p)=0\}$ be  set of equilibrium points of  $(M,g,f)$. Suppose that
 scalar curvature $R$ of $g$ decays uniformly.  Then the following statements are true.

 \begin{enumerate}
 \item[(1)]$(S,g_S)$ is a compact flat manifold, where $g_{S}$ is an induced metric $g$.
 \item[(2)]Let $o\in S$.  Then level set $\Sigma_r=\{x\in M|f(x)=r\}$ is a closed hypersurface of $M$. Moreover,  each  $\Sigma_r$ is   diffeomorphic to each other  whenever  $r>f(o)$.
\item[(3)] $M_r=\{x\in M| ~f(x)\le r\}$ is  compact for any $r>f(o)$.
\item[(4)] $f$ satisfies (\ref{inequality-cao chen}).

 \end{enumerate}

\end{lem}

\begin{proof}
(1)  Let  $(\widetilde{M},\widetilde{g})$ be the universal cover  of $(M,g,f)$ with the covering map $\pi$. Let $\widetilde{f}=f\circ\pi$.
It is  clear that $(\widetilde{M},\widetilde{g}, \widetilde{f})$ is also a  steady gradient Ricci soliton.
Then by \cite[Theorem 1.1]{GLX}, there is an $(n-k)$-dimensional  steady gradient Ricci soliton  $(N,h, f_N)$ with nonnegative sectional curvature and positive Ricci curvature such that $(\widetilde{M},\widetilde{g})=(N,h)\times\mathbb{R}^k$ $(k\ge 0$).    Let $(q,y')$ be a coordinate system  on $\widetilde{M}=N\times\mathbb{R}^k$, where $q\in N$ and
$y'=( y_1,\cdots,y_k)\in \mathbb{R}^k.$  We claim
\begin{align}\label{f=constant}
\frac{\partial\widetilde{f}}{\partial y_j}=0,~\forall~1\le j\le k.
\end{align}

Fix $q \in N$ and $y_i$ with  $i\neq j$.  Let
\begin{align}
\widetilde{f}_j(y)=\widetilde{f}(q, y_1,\cdots, y_{j-1},y,y_{j+1},\cdots,y_k),~\forall  ~y\in \mathbb R.\notag
\end{align}
By the Ricci soliton equation, we have
\begin{align}
\frac{\partial^2 \widetilde{f}_j}{\partial y^2}=\widetilde{\rm Ric}(\frac{\partial}{\partial y},\frac{\partial}{\partial y})=0.\notag
\end{align}
It follows
\begin{align}\label{f=constant-2}
\widetilde{f}_j(y)=c_1y+c_2,
\end{align}
where $c_1$ and $c_2$ are constants.
Thus
\begin{align}\label{linear-y}
|\widetilde{f}_j(y)-\widetilde{f}_j(-y)|=2|c_1y|, ~\forall ~y\in \mathbb R.
\end{align}

We define a set
$$E=\{p\in M|~p=\pi(q, y_1,\cdots, y_{j-1},y, y_{j+1},\cdots,y_k),~ y\in\mathbb{R}\}.$$
Then for any $p\in E$, we have
\begin{align}
R(p)=\widetilde{R}(q, y_1,\cdots, y_{j-1},y,y_{j+1},\cdots,y_k)=R_{h}(q)>0.\notag
\end{align}
Thus
\begin{align}
E\subseteq \{p\in M|~R(p)\le R_{h}(q)\}=E'.\notag
\end{align}
Since $R$ decays uniformly,  the set $E'$ is bounded and so is $E$.  Hence
\begin{align}
{\rm diam}(E,g)=D<+\infty.\notag
\end{align}
Note that $|\nabla f|\le \sqrt{R_{\max}}$. Therefore,  for any $p_1,p_2\in E$, we integrate from $p_2$ to $p_1$ along a minimal geodesic to get
\begin{align}
f(p_1)-f(p_2)\le \sqrt{R_{\max}}d(p_1,p_2)\le D\sqrt{R_{\max}}. \notag
\end{align}
Choose $p_1=\pi(q, y_1,\cdots, y_{j-1},m,y_{j+1},\cdots,y_k)$ and $p_2=\pi(q, y_1,\cdots, y_{j-1},$ $-m,y_{j+1},\cdots,y_k)$.
By (\ref{linear-y}), we derive
\begin{align}
2|c_1m|=|\widetilde{f}_j(m)-\widetilde{f}_j(-m)|=|f(p_1)-f(p_2)|\le D\sqrt{R_{\max}},\notag
\end{align}
As a consequence,  $c_1=0$ by taking $m\to\infty$.  This implies  (\ref{f=constant})  by (\ref{f=constant-2}).

 By (\ref{f=constant}), we may assume that  $f_N(q)=\widetilde{f}(q,\cdot)$. Since $R(p)$ attains  its maximum in $M$, $R_h(q)$ attains its maximum at some point $o_N\in N$. Note that  ${\rm Ric}(h)$ is positive.  Then  by an argument in \cite[Corollary 2.2]{DZ5},  we  see  $\nabla_{h}f_N (o_N)=0$. Moreover, such a  $o_N$ is unique.  Thus
 \begin{align}\label{s-set}
\pi^{-1}(S)=\{o_N\}\times\mathbb{R}^k.
\end{align}
Since  $S=\{p\in M|~R(p)=R_{\max}\}$ is compact by the curvature decay,  $(S,g_{S})$ is a compact flat manifold.

(2) Let  $o\in S$. Then by (\ref{s-set}), we have
$f(S)\equiv f(o)$.   By (\ref{identity}), it follows that  for any $r>f(o)$,
  $$|\nabla f|^2(p)=R_{{\rm max}}-R(p)>0, ~\forall~ p\in \Sigma_r.$$
  Thus $\Sigma_r$ is a hypersurface of $M$.  In the following,  we show that it is bounded.

  Choose $q\in N$ such that  $\pi(q,y)=p\in \Sigma_r$. Then  $f_N(q)=r$. Let $(\phi_N)_t$ be  a one-parameter diffeomorphisms generated by $-\nabla_{h} f_N$. Thus
\begin{align}\label{convergence-map-1}
d_{h}((\phi_N)_t(q),o_N)\to0,~{\rm as}~t\to\infty.
\end{align}
Moreover,  the above  convergence is uniform for all $q\in \{x\in N|~f_N(x)=r\}$. Similarly,  we have
\begin{align}\label{convergence-map-2}
\widetilde{d}(\widetilde{\phi}_t(q,y),(o_N,y))\to0,~{\rm as}~t\to\infty,
\end{align}
where  $\widetilde{\phi}_t$ is a one-parameter diffeomorphisms  generated by $-\widetilde{\nabla}\widetilde{f}$ and
the convergence of (\ref{convergence-map-2}) is uniform on $\{\tilde x\in \widetilde{M}| ~\widetilde{f}(\tilde x)=r\}$. Note that $\pi(\widetilde{\phi}_t(q,y))=\phi_t(p)$. Thus
\begin{align}
d(\phi_t(p),\pi(o_N,y))\le \widetilde{d}(\widetilde{\phi}_t(q,y),(o_N,y))\to 0,~{\rm as}~t\to\infty.\notag
\end{align}
 It follows that
\begin{align}\label{convergence-map-3}
d(\phi_t(p),S)\to 0,,~{\rm as}~t\to\infty.
\end{align}
Moreover, the convergence is uniform on $\Sigma_r$.

By (\ref{convergence-map-3}),  there is a sufficiently large  $t_0$ such that
\begin{align}\label{s-distance}
d(\phi_{t}(\Sigma_r),S)\le 1,~\forall~t\ge t_0.
\end{align}
Let $\gamma_p(s)=\phi_s(p)$, $s\in[0,t_0]$. Then,
\begin{align}\label{length of phi_t}
d(p,\phi_{t_0}(p))\le {\rm Length}(\gamma_p,g)=\int_{0}^{t_0}|\nabla f|(\phi_t(p)){\rm d}s\le t_0\sqrt{R_{\max}}.
\end{align}
It follows that
\begin{align}
d(\Sigma_r,S)\le d(\Sigma_r,\phi_{t_0}(\Sigma_r))+d(\phi_{t_0}(\Sigma_r),S)\le t_0\sqrt{R_{\max}}+1.\notag
\end{align}
Hence $\Sigma_r$ is bounded since $S$ is compact.
$\Sigma_{r_1}$ and $\Sigma_{r_2}$ are diffeomorphic to each other for all $r_1,r_2>f(o)$
by the fact  $|\nabla f|(x)>0$ for all $f(x)> f(o)$.

(3) Since $\Sigma_r$ is a closed set,  it suffices to show that the set  $M_r'=\{f(o)<f(x)<r\}$  is bounded by the above  properties (1) and (2).  For any $x\in M_r'$, choose a point $x'\in\Sigma_r$ and a number $t_x>0$ such that $\phi_{t_x}(x')=x$. If $t_x\ge t_0$, then $d(x,S)\le 1$ by (\ref{s-distance}). If $t_x<t_0$, then $d(x,\Sigma_r)\le d(x,x')\le t_0\sqrt{R_{\max}}$ by (\ref{length of phi_t}). Thus
\begin{align}
d(x,S)\le d(x,\Sigma_r)+d(\Sigma_r,S)\le 2(t_0\sqrt{R_{\max}}+1).
\end{align}
Hence,  $M_r'$ is bounded.

(4)  Note that  $(N,h,f_N)$ has positive Ricci curvature. Then similar to  (\ref{inequality-cao chen}), we have
\begin{align}
f_N(q)\ge Cd_h(q,o_N),~\forall~f_N(q)\ge r_0.\notag
\end{align}
It follows that
\begin{align}
\widetilde{f}(\tilde x)\ge C\widetilde{d}(\tilde x,\{o_N\}\times\mathbb{R}^k),~\widetilde{f}(\tilde x)\ge r_0,\notag
\end{align}
where $\pi(\widetilde{x})=x$, Thus
\begin{align}
f(x)=\widetilde{f}(\tilde x)\ge C\widetilde{d}(\widetilde{x},\{o_N\}\times\mathbb{R}^k)\ge Cd(x,S),~\forall~f(x)\ge r_0,\notag
\end{align}
As a consequence, we get
\begin{align}
f(x)\ge C\rho(x)-CC_S,~\forall~f(x)\ge r_0,\notag
\end{align}
where $C_S={\rm diam}(S,g)$.
Since $\{x\in M|~\rho(x)\le k\}_{k\in\mathbb{N}}$ exhaust $M$ as $k\to\infty$ and $M_{r_0}$ is compact by the above  property (3), there exists a constant $r_0' $ $(\gg r_0)$ such that
\begin{align}
M_{r_0}\subset \{x\in M|~\rho(x)<r_0'\}.\notag
\end{align}
Namely,
\begin{align}
\{x\in M|~  \rho(x)\ge r_0'\}\subset \{x\in M|~f(x)\ge r_0\}\notag
\end{align}
Hence
\begin{align}
 f(x)\ge C\rho(x)-CC_S,~\forall~\rho(x)\ge r_0'.\notag
\end{align}
Therefore, we get the left side of  (\ref{inequality-cao chen}). The right side follows from the fact
\begin{align}
f(x)\le f(o)+\sqrt{R_{\max}}\rho(x),~\forall~x\in M.\notag
\end{align}
The lemma is proved.

\end{proof}

\begin{rem}
\begin{enumerate}
\item[(1)]  It is possible that $\Sigma$ is empty for  steady Ricci soltions with nonnegative sectional curvature. For example, $(\mathbb{R}^n,g_{Euclid}$, $f=\Sigma_{i=1}^nx_i)$ is a steady Ricci soliton with $|\nabla f|^2\equiv n$.

\item[(2)] The estimate $\widetilde{f}(x)\ge C\widetilde{\rho}(x)$, ~$\widetilde{\rho}(x)\ge r_0$ fails on the universal cover  $(\widetilde{M},\widetilde{g})$ of $(M,g)$  in Lemma \ref{lem-level set structure nonnegative case},  since $\widetilde{R}(x)$ doesn't decay uniformly.

\end{enumerate}

\end{rem}

\begin{cor}\label{cor-lower decay nonnegative case}
Let $(M,g, f)$ be a non-flat  $\kappa$-noncollapsed steady Ricci soliton  with nonnegative curvature operator and uniform curvature decay.  Then   scalar curvature  of $g$ satisfies
(\ref{lower-bound-r}).
\end{cor}

\begin{proof}Let  $(\widetilde{M}, \tilde g, \tilde f)$  be the covering   steady Ricci soliton of  $(M,  g, f)$ as in  Lemma \ref{lem-level set structure nonnegative case}.  Then $\widetilde{M}=N\times \mathbb R^k$  is also $\kappa$-noncollapsed, and   $(N,h,f_N)$ is a $\kappa$-noncollapsed steady gradient Ricci soliton with nonnegative curvature operator and positive Ricci curvature, where $f_N(q)=\tilde f(q,\cdot)$.  Moreover,  $(N,h,f_N)$ admits a unique equilibrium point $o_N$.  Thus by Theorem \ref{lower-bound-scalar-curvature}, we have
\begin{align}
R_N(q)f_N(q)\ge C_0,~\forall~f_N(q)\ge r_0,~q\in N.\notag
\end{align}
It follows that
\begin{align}
\widetilde{R}(x)\widetilde{f}(x) \ge C_0,~\forall~\widetilde{f}(\tilde x)\ge r_0,~\tilde x\in \widetilde{M},\notag
\end{align}
and
\begin{align}
R(x)f(x)\ge  C_0,~\forall~f(x)\ge  r_0,~x\in M.\notag
\end{align}
Combining the above  with (\ref{inequality-cao chen}), we get (\ref{lower-bound-r}) immediately.

\end{proof}

With the help of  (2)-(4) in Lemma \ref{lem-level set structure nonnegative case}, we can extend the  arguments in Section 2-4 to prove
a weak version of Theorem \ref{main theorem-2}.

\begin{theo}\label{theo-convergence nonnegative case}
Let $(M,g)$ be a  $\kappa$-noncollapsed steady Ricci soliton  with nonnegative sectional curvature.  Suppose that  $(M,g)$  satisfies (\ref{linear-decay-curvature}).   Then, for any $p_{i}\rightarrow \infty$,
 rescaled flows $(M,R(p_i)g(R^{-1}(p_i)t),p_{i})$ converge   subsequently to
$(\mathbb{R}\times \Sigma,$ $ds^2+g_{\Sigma}(t))$ ( $t\in (-\infty,0]$) in the Cheeger-Gromov topology, where   $\Sigma$ is diffeormorphic to  a  level set $\Sigma_{r_0}$ in    $(M,g)$ and
 $(\Sigma,$ $g_{\Sigma}(t))$ is a $\kappa$-noncollapsed ancient Ricci flow with nonnegative sectional curvature. Moreover,  scalar curvature $R_{\Sigma}(x,t)$ of $(\Sigma,g_{\Sigma}(t))$ satisfies
 \begin{align}\label{decay of t}
 R_{\Sigma}(x,t)\le\frac{C}{|t|}, ~\forall~x\in \Sigma,
 \end{align}
 where $C$ is a uniform constant.
\end{theo}

The proof of  Theorem \ref{theo-convergence nonnegative case} is almost the same as one of Theorem \ref{main theorem-strong} by replacing $\mathbb S^{n-1}$ with $\Sigma$. We leave it to the readers.

\begin{proof}[Proof of Theorem \ref{main theorem-2}] By Theorem \ref{main theorem-strong}, it suffices to show that  $(M,g)$ has positive Ricci curvature.  Otherwise, we assume that the Ricci curvature is not strictly positive.
Let  $(\widetilde{M}, \tilde g, \tilde f)$  be the covering   steady Ricci soliton of  $(M,  g, f)$ as in  Lemma \ref{lem-level set structure nonnegative case}.  Then  $\widetilde{M}$ splits off   $\mathbb{R}^k$ $(k\ge 1$) as $N\times\mathbb{R}^k$, where  $(N,h)$ has positive Ricci curvature.  Let $\widetilde{V}\in T\widetilde{M}$ be a vector field parallel to $\mathbb{R}^k$ and $|\tilde V|_{\widetilde{g}}\equiv1$. Then $\widetilde{g}(\widetilde{V},\widetilde{\nabla}\widetilde{f})=0$ since $\widetilde{f}|_{\mathbb{R}^k}\equiv {\rm const.}$  Let $\{p_i\}$  be a sequence of points in $M$  with $f(p_i)\to \infty$,  and $\{\widetilde{p}_i\}$ in $\widetilde{M}$  with  $\pi(\widetilde{p}_i)=p_i$. Let $W_{(i)}=R(p_i)^{-1/2}\pi_{\ast}(V(\widetilde{p}_i))\in T_{p_i}M$ and $X_{(i)}=R(p_i)^{-1/2}\nabla f$. Then
\begin{align}
{\rm Ric}_{g(t)}(W_{(i)},W_{(i)})=0,~|W_{(i)}|_{R(p_i)g}\equiv1,~\langle W_{(i)},X_{(i)}(p_i)\rangle=0,~\forall~i,~t\le0.\notag
\end{align}
Since $(M,R(p_i)g(R^{-1}(p_i)t),p_{i})$ converges subsequently to
$(\mathbb{R}\times \Sigma,$ $ds^2+g_{\Sigma}(t),p_{\infty})$  by Theorem \ref{theo-convergence nonnegative case},  as in the proof of Lemma \ref{lem-convergence and splitting}, we get
$X_{(i)}\to X_{(\infty)}$, where the limit vector field  $X_{(\infty)}$ is parallel to $\mathbb{R}$   in  $\mathbb{R}\times \Sigma$.
Moreover
\begin{align}\label{orthogonal-1}
W_{(i)}\to W_{(\infty)},~\langle W_{(\infty)},X_{(\infty)}(p_{\infty})\rangle=0,
\end{align}
and
\begin{align}\label{orthogonal-2}
|W_{(\infty)}|_{g_{\Sigma}(0)+ds^2}=1,~{\rm Ric}_{g_{\Sigma}(t)+ds^2}(W_{(\infty)},W_{(\infty)})=0,~\forall~t\le 0.
\end{align}

 Since $( \Sigma, g_{\Sigma}(t))$ satisfies (\ref{decay of t}),
 by Theorem 3.1 in \cite{Na},  we see that $(\Sigma,\tau_i g_{\Sigma}(\tau_i^{-1} t),q)$
  subsequently converges to a shrinking Ricci soliton   $(\Sigma_{\infty},$ $g_{\Sigma_{\infty}}(t),q_{\infty})$  for any fixed $q\in \Sigma$ and  any sequence $\{\tau_i\}\to\infty$.  On the other hand, by a result in \cite{Ni},  there exists a constant $C_1$  such that
 \begin{align}
 {\rm diam}( \Sigma, g_{\Sigma}(t))\le C_1\sqrt{|t|}.\notag
 \end{align}
 In particular,
 \begin{align}\label{diam estimate}
 {\rm diam}(  \Sigma, \tau_i^{-1} g_{\Sigma}(-\tau_i))\le C_1.
 \end{align}
    Thus,  by Cheeger-Gromov compactness theorem,  $\Sigma_{\infty}$ is diffeomorphic to $\Sigma$.
  Choose $\tau_i'$ such that $|\tau_i' Q|_{h_i}=1$, where $h_i=\tau_i g_{\Sigma}(-\tau_i^{-1})$.
   Assume that  $p_{\infty}=(x,q)\in\mathbb R\times \Sigma$ and
  $W_{(\infty)}=Q+Q'$, where $Q\in T_q\Sigma$ and $Q'\in T_x\mathbb{R}$. Then by (\ref{orthogonal-1}) and  (\ref{orthogonal-2}), we have
 $Q'=0$ and
\begin{align}
|Q|_{g_{\Sigma}(0)}=1,~{\rm Ric}_{g_{\Sigma}(t)}(Q,Q)=0,~\forall~t\le0.\notag
\end{align}
It follows that $\tau_i'Q\to Q_{\infty}$,  where $Q_{\infty}\in T_{q_{\infty}}\Sigma_{\infty}$.  As a consequence,
\begin{align}
|Q_{\infty}|_{g_{\Sigma_{\infty}}(-1)}=1,~{\rm Ric}_{g_{\Sigma_{\infty}}(t)}(Q_{\infty},Q_{\infty})=0,~\forall~t\le 0.\notag
\end{align}
 Hence,  $(\Sigma_{\infty},g_{\Sigma_{\infty}}(t))$ is a compact shrinking Ricci soliton  with nonnegative sectional curvature,  but not strictly positive  Ricci curvature.  By  \cite[Theorem 1.1]{GLX},  the universal cover of $(\Sigma_{\infty},g_{\Sigma_{\infty}}(t))$ must split  off a flat factor $\mathbb{R}^{l}$ $(l\ge1)$.  On the other hand,  since the fundamental group of any compact shrinking Ricci soliton is finite \cite[Theorem 1]{L} (also see \cite{Ma}, \cite{CTZ}),   the universal cover of $\Sigma_\infty$ should be comapct. Therefore, we get a contradiction.  The proof is completed.

\end{proof}

\section{Proofs of Theorem \ref{main theorem} and  Theorem \ref{cor-rotational symmetry of 4d}}

\begin{proof}[Proof of Theorem \ref{main theorem}] We may assume that  the steady (gradient) Ricci soliton $(M,g)$ is not falt.  By  the condition (\ref{upper-linear-bound}) in Theorem \ref{main theorem} together with  Corollary  \ref{cor-lower decay nonnegative case},  we see that Theorem \ref{main theorem-2} is true for  a  $\kappa$-noncollapsed  $(M,g)$  with nonnegative curvature operator  if $(M,g)$ satisfies (\ref{upper-linear-bound}).  Then
the  ancient solution  $g_{\mathbb{S}^{n-1}}(t)$  in Theorem \ref{main theorem-2} is in fact a compact  $\kappa$-solution which satisfies (\ref{curvature-decay-sn}). By a result of Ni \cite{Ni},  $g_{\mathbb{S}^{n-1}}(t)$  must be a  flow of  shrinking round spheres. This means that    $(M,g)$ is   asymptotically   cylindrical.  Hence by   Brendle's result \cite{Br2},  it is  rotationally symmetric.

\end{proof}

\begin{proof}[Proof of Theorem  \ref{cor-rotational symmetry of 4d}]  We note that the  ancient solution  $g_{\mathbb{S}^{n-1}}(t)$  in  Theorem  \ref{main theorem-2} satisfies the condition (\ref{curvature-decay-sn}).   Then,  as in   the proof of Theorem  \ref{main theorem-2}),     $(\mathbb{S}^{n-1},\tau_i^{-1} g_{\mathbb{S}^{n-1}}(\tau_i t),x)$ subsequently converges  to a  compact gradient shrinking Ricci soliton $(N',$ $g'(t),x')$ for any fixed $x\in \mathbb{S}^{n-1}$ and  any sequence $\{\tau_i\}\to\infty$   Moreover,   $N'$ is diffeomorphic to $\mathbb{S}^{n-1}$.   Note that $n=4$. By \cite{H2} or \cite{BW},  any shrinking Ricci soliton with nonnegative sectional curvature on $\mathbb{S}^3$ must be a round sphere. Hence,
$(N',g'(t),x')$ is a flow of  shrinking round  spheres.  This implies that  $(\mathbb{S}^3,\tau_i^{-1}g_{\mathbb{S}^3}(\tau_i t_0))$ has a  strictly positive curvature operator as long as $\tau_i$ is sufficiently  large. As a consequence,  $(\mathbb{S}^{3},g_{\mathbb{S}^{3}}(\tau_i t_0))$ has positive curvature operator.  Since the positivity of  curvature operator is preserved under Ricci flow,   $g_{\mathbb{S}^{3}}(t)$ has positive curvature operator for all $t\in(-\infty,0)$.  Therefore, we see that the  ancient solution  $g_{\mathbb{S}^{3}}(t)$ is  a compact  $\kappa$-solution which satisfies (\ref{curvature-decay-sn}). As  in the proof of  Theorem \ref{main theorem},  $(M,g)$ is   asymptotically   cylindrical, and so it is   rotationally symmetric by    Brendle's result  \cite{Br2}.

\end{proof}

\subsection{Further remarks}

The rigidity of nonnegatively curved $n$-dimensional $\kappa$-noncollapsed steady  (gradient) Ricci solitons is closely related to the classification of positively  curved shrinking Ricci soliton on $\mathbb{S}^{n-1}$. To the authors'  knowledge, it is still  unknown whether the shrinking soliton with positive  sectional curvature on $\mathbb{S}^{n}$ is unique when $n\ge4$. If it is true, then   the argument in proof of Theorem \ref{cor-rotational symmetry of 4d} can be generalized to any dimension. Namely, we may prove that
any $n$-dimensional steady  (gradient)  Ricci soliton with   nonnegative  sectional curvature, positive Ricci curvature and exactly linear curvature decay must be rotationally symmetric.


\section*{References}


\vskip6mm
\small

\begin{enumerate}


\renewcommand{\labelenumi}{[\arabic{enumi}]}

\bibitem{BW} B\"{o}hm, C. and Wilking, B., \textit{Manifolds with positive curvature operators are space forms}, Ann. of Math., \textbf{167} (2008), 1079-1097.

%\bibitem{Chow} Chow, B.; Chu, S.C.; Glickenstein, D.; Guenther, C.; Isenberg, J.; Ivey, T.; Knopf, D.; Lu, P.; Luo, F.; Ni, L., The Ricci flow: %techniques and applications. Part IV. Long-time solutions and related topics. Mathematical Surveys and Monographs, 206. American Mathematical %Society, Providence, RI, 2015. xx+374. ISBN: 978-0-8218-4991-0.

\bibitem{Br1} Brendle, S., \textit{Rotational symmetry of self-similar solutions to the Ricci flow}, Invent. Math. , \textbf{194} No.3 (2013), 731-764.

\bibitem{Br2} Brendle, S., \textit{Rotational symmetry of Ricci solitons in higher dimensions}, J. Diff. Geom., \textbf{97} (2014), no. 2, 191-214.





\bibitem{CaCh} Cao, H.D. and Chen, Q., \textit{On locally conformally flat gradient steady Ricci solitons},
Trans. Amer. Math. Soc., \textbf{364} (2012), 2377-2391 .

\bibitem{CTZ} Cao, H., Tian, G. and Zhu, X.H., \textit{K\"ahler-Ricci solitons on compact complex manifolds with $c_1(M)>0$},   Geom. Funct. and Anal., \textbf{15} (2005),  697-719.

%\bibitem{CG} Cheeger, J. and Gromoll, D., \textit{The splitting theorem for manifolds of nonnegative Ricci curvature}, J. Diff. Geom., \textbf{6} (1972), 119-128.



%\bibitem{CGT} Cheeger, J., Gromov, M. and Taylor, M., \textit{Finite propagation speed, kernel estimates for functions of the Laplace operator, and the geometry of complete Riemannian manifolds}, J. Diff. Geom., \textbf{17} (1982), 15-53.

%\bibitem{CZ} Chen, B.L. and Zhu X.P., \textit{Volume growth and curvature decay of positively curved K\"{a}hler manifolds},
%Q. J. Pure Appl. Math. , \textbf{1} (2005), no. 1, 68-108.


\bibitem{Ch} Chen, B.L., \textit{Strong uniqueness of the Ricci flow},  J. Diff.  Geom. \textbf{82} (2009),  363-382.



\bibitem{DZ1}Deng, Y.X. and Zhu, X.H., \textit{Complete non-compact gradient Ricci solitons with nonnegative Ricci curvature},
Math. Z., \textbf{279} (2015), no. 1-2, 211-226.

\bibitem{DZ2} Deng, Y.X. and Zhu, X.H., \textit{Asymptotic behavior of positively curved steady Ricci solitons}, arXiv:math/1507.04802, to appear in Trans. Amer. Math. Soc..

 \bibitem{DZ3} Deng, Y.X. and Zhu, X.H.,  \textit{ Steady Ricci solitons  with horizontally $\epsilon$-pinched  Ricci curvature }, arXiv:math/1601.02111.

\bibitem{DZ5} Deng, Y.X.; Zhu, X.H., 3D steady gradient Ricci solitons with linear curvature decay, arXiv:math/1612.05713, to appear in IMRN.

\bibitem{DZ4} Deng, Y.X. and Zhu, X.H., \textit{Asymptotic behavior of positively curved steady Ricci solitons, II}, arXiv:math/1604.00142.


%\bibitem{EMT} Enders, J.; Muller, R.;Topping, P.M., On type-I singularities in Ricci flow. Comm. Anal. Geom. 19 (2011)£¬905-922.


\bibitem{GLX} Guan, P.F., Lu, P. and Xu, Y.Y., \textit{A regidity theorem for codimension one shrinking gradient Ricci solitons in $\mathbb{R}^{n+1}$}
 Calc. Var. Partial Differential Equations \textbf{54} (2015), no. 4, 4019-4036.

\bibitem{H2} Hamilton, R.S., \textit{Three manifolds with positive Ricci curvature}, J. Diff. Geom., \textbf{17} (1982), 255-306.

\bibitem{H4} Hamilton, R.S., \textit{A compactness property for solution of the Ricci flow}, Amer. J. Math., \textbf{117} (1995), 545-572.

\bibitem{H1} Hamilton, R.S., \textit{Formation of singularities in the Ricci flow}, Surveys in Diff. Geom., \textbf{2} (1995),
7-136.

\bibitem{L} Lott, J., \textit{Some geometric properties of the Bakry-\'{E}mery-Ricci tensor}. Comment. Math. Helv. 78 (2003), no. 4, 865-883.

\bibitem{Ma} Mabuchi, T., \textit{Heat kernel estimates and the Green functions on Multiplier Hermitian manifolds }.  Toh\"oku Math. J. 54 (2002), 259-275.

\bibitem{MT} Morgan, J. and Tian, G., \textit{ Ricci flow and the Poincar\'{e} conjecture}, Clay Math. Mono., 3. Amer. Math. Soc., Providence, RI; Clay Mathematics Institute, Cambridge, MA, 2007, xlii+521 pp. ISBN: 978-0-8218-4328-4.



\bibitem{Ni} Ni, L., \textit{Closed type-I Ancient solutions to Ricci flow}, Recent Advances in Geometric Analysis, ALM, vol. 11 (2009), 147-150.

\bibitem{Na} Naber, A., \textit{Noncompact shrinking four solitons with nonnegative curvature}, J. Reine Angew Math., \textbf{645} (2010), 125-153.


\bibitem{Pe} Perelman, G., \textit{The entropy formula for the Ricci flow and its geometric applications}, arXiv:math/0211159.

%\bibitem{Pe2} Perelman, G., \textit{Ricci
%flow with surgery on three-manifolds},\\ arXiv:math/0303109v1.



\bibitem{S1} Shi, W.X., \textit{Ricci deformation of the metric on complete noncompact Riemannian
manifolds}, J. Diff. Geom., \textbf{30} (1989), 223-301.


\end{enumerate}
\end{document}